\documentclass[12pt]{amsart}
\usepackage[colorlinks=true,citecolor=green,linkcolor=magenta]{hyperref}
\usepackage{amsmath}
\usepackage{verbatim}
\usepackage{amsfonts}
\usepackage{amssymb}
\usepackage{blkarray}
\usepackage{color,colortbl}
\usepackage[all]{xy}  
\usepackage{enumerate}
\usepackage[top=1in, bottom=1in, left=1in, right=1in]{geometry}
\usepackage{mathrsfs}

\usepackage{stmaryrd}
\usepackage{bbold}
\renewcommand{\k}{\mathbb{k}}
\usepackage{cleveref}
\usepackage{soul}
\usepackage{arydshln,xcolor}
\usepackage{tikz-cd}
\usepackage{graphicx}
\theoremstyle{definition}
\newtheorem{theorem}{Theorem}[section]
\newtheorem{theoremx}{Theorem}

\numberwithin{equation}{section}

\newtheorem{lemma}[theorem]{Lemma}

\theoremstyle{definition}
\newtheorem{definition}[theorem]{Definition}

\newtheorem{example}[theorem]{Example}

\newtheorem{remark}[theorem]{Remark}

\newtheoremstyle{TheoremNum}
{8pt}{8pt}              
{\upshape}                      
{}                              
{\bfseries}                     
{.}                             
{.5em}                             
{\theoremname{#1}\theoremnote{ \bfseries #3}}
\theoremstyle{TheoremNum}

\newcommand{\m}{\mathfrak{m}}

\newcommand{\CC}{\mathbb{C}}

\newcommand{\cC}{\mathfrak{C}}

\newcommand{\Rank}{\operatorname{rank}}

\newcommand{\Hom}{\operatorname{Hom}}

\newcommand{\C}{\mathfrak{C}}

\newcommand{\quot}{\operatorname{Quot}}

\renewcommand{\leq}{\leqslant}
\renewcommand{\geq}{\geqslant}
\newcommand{\ds}{\displaystyle}

\newcommand{\ps}[1]{\llbracket {#1} \rrbracket}

\newcommand{\h}{\operatorname{h}}
\newcommand{\hun}{\displaystyle \h(\Omega_{R})}
\newcommand{\Tr}{\text{tr}_R}

\title{Two Criteria for Quasihomogeneity}

\address{Department of Mathematics, University of Utah, Salt Lake City, UT, USA}

\author[Maitra]{Sarasij Maitra}
\email{maitra@math.utah.edu}
\address{Department of Mathematics, Indian Institute of Technology Delhi, Hauz Khas, India.}

\author[Mukundan]{Vivek Mukundan}
\email{vmukunda@iitd.ac.in}

\begin{document}

	\begin{abstract}
Let $(R,\m_R,\k)$ be a one-dimensional complete local reduced $\k$-algebra over a field of characteristic zero. The ring $R$ is said to be quasihomogeneous  if there exists a surjection $\Omega_R\twoheadrightarrow \m$ where $\Omega_R$ denotes the module of differentials. We present two characterizations of quasihomogeniety of $R$. The first one on the valuation semigroup of $R$ and the other on the trace ideal of the module $\Omega_R$.
	\end{abstract}
	\maketitle

 \section{Introduction}
Let $(R,\m,\k)$ be a one dimensional complete local domain where $\k$ is algebraically closed field of characteristic zero. We can identify $R$ with $\k\ps{X_1,\dots,X_n}/I$ where $I$ is called the defining ideal of $R$. Scheja in \cite{scheja1970differentialmoduln} defines $R$ to be \textit{quasihomogeneous} if there exists a surjection $\Omega_R\twoheadrightarrow \m$ where $\Omega_R$ is the module of differentials. When $I$ is generated by polynomials, this definition coincides with the standard definition of quasihomogeneity found in literature:  i.e., $I$ is generated by quasihomogeneous polynomials (\cite[Satz 2.1]{KunzRuppert77}). A polynomial $f$ is said to be quasihomogeneous of degree $d$, if there exists weights $w_1,\dots,w_n\in\mathbb{N}$ such that $f(\lambda^{w_1}x_1,\dots,\lambda^{w_n}x_n)=\lambda^d f(x_1,\dots,x_n)$ for $\lambda\in\k$.  Scheja  proved  Berger's conjecture (\cite{Berger_article}) for quasihomogeneous rings in \cite{scheja1970differentialmoduln}. A more general version of this result appeared recently in \cite[Theorem 4.11]{huneke2021torsion}. Thus identifying when the ring $R$ is \textit{quasihomogeneous} is indeed important and it is desirable to get easy methods to recognize the quasihomogeneous property for $R$. One of the first results in this direction is due to Zariski (\cite{Zariski66}) who showed that for  irreducible plane algebroid curves, $R$ is quasihomogeneous when the module of differentials $\Omega_R$ has maximal torsion. Saito (in \cite{Saito71}) showed that if $R$ has an isolated hypersurface singularity ($I=(f)$) then $R$ is quasihomogeneous when $f$ is in the ideal generated by the partial derivatives $\frac{\partial f}{\partial x_1},\dots,\frac{\partial f}{\partial x_n}$.  The most often used characterization of quasihomogeneous rings defined by polynomials, is due to Kunz and Ruppert \cite[Satz 3.1]{KunzRuppert77}. They showed that $R$ is quasihomogeneous if and only if $R$ is isomorphic to a numerical semigroup ring $\k\ps{H}$ where $H$ is a numerical semigroup. 

In this article, we present two new characterizations of quasihomogeneity of $R$. For the first characterization we make use of the conductor ideal $\C_R=\overline{R}:_{\quot(R)}R$ where $\quot(R)$ is the quotient field of $R$ and $\overline{R}$ is the integral closure of $R$ in $\quot(R)$.  Since $\overline{R}=\k\ps{t}$ we can identify $R$ with $\k\ps{t^{a_1},\dots,t^{a_n}}$. We can define an order valuation $v(\sum\alpha_it^i)=\min\{ j~|~a_j\neq 0\}$. A subsequent valuation can be defined as $o(\sum_{i=0}^\infty\alpha_it^i)=v\left(\sum_{i=0}^\infty\alpha_it^i-\alpha_0\right)$. We use this valuation as a check for quasihomogeneity.
 \begin{theoremx}\label{quasihomogeneous and valuations}
	Let $(R,\m,\k)$ be a non-regular, complete, local one dimensional domain which is a $\k$-algebra with $\k$ algebraically closed of characteristic $0$. Let $\overline{R}=\k\ps{t}$ with the conductor of $R$ given by $\mathfrak{C}_R=(t^{c_R})\overline{R}$. Writing $R=\k\ps{\alpha_1t^{a_1},\ldots, \alpha_nt^{a_n}}$ with $a_1<a_2<\cdots<a_n$, set 
	\begin{align*}
	o(\alpha_r)=\min_{1\leq i\leq n}\{o(\alpha_i) \}, \qquad a=\min_{j\neq r}\{a_j\}.
	\end{align*}
	Then $R$ is quasihomogeneous  if $o(\alpha_r)+a\geq c_R$.
\end{theoremx}
This provides a quick way to check quasihomogeneity of $R$. Of course, the downside is that we need to have information on the valuation $c_R$ which, in general can be hard to compute. But this result gives a quick numerical characterization of quasihomogeneity and easily deployable in computational packages such as Macaulay2.

For the second characterization, we use the invariant defined in \cite{maitra2020partial}. We use the description of the trace ideal, $\Tr(\Omega_R)=\sum \alpha(\Omega_R)$ where $\alpha\in\Omega_R^*=\Hom_R(\Omega_R,R)$. Since the valuation function $v(\cdot)$ can be extended to all fractional ideal in $\overline{R}$, we get another characterization of the quasihomogeneity using the valuation of the trace ideal $\Tr(\Omega_R)$. Since $R$ is quasihomogeneous, we have $v(\m)\geq v(\Tr(\Omega_R))$, by definition. In fact since $R$ is non-regular, this is an equality. However, the converse is unclear (see \cite[Conclusion 3]{maitra2020partial}). We settle this completely in the following result.
\begin{theoremx}
    Let $(R,\m,\k)$ be a non-regular complete local one dimensional domain with $\k$ algebraically closed of characteristic $0$. Let $\Omega_R$ be the module of differentials and $\overline{R}=\k\ps{t}$ with valuation function $v$. Then $R$ is quasihomogeneous if and only if $v(\Tr(\Omega_R))=v(\m).$
\end{theoremx}
\subsection*{Acknowledgements}  We thank Prof. Craig Huneke for discussions regarding both the main results in the article. We deeply thank the referee for pointing out various improvements (especially Theorem 3.1) to the article.
 \section{Preliminaries}
Throughout this  article, $(R,\m,\k)$ will denote a non-regular complete one dimensional local domain which is a $\k$-algebra with $\k$ algebraically closed of characteristic $0$. Hence $R={P}/{I}$ where $P=\k\ps{X_1,\cdots, X_n}, n\geq 2$ and $I\subseteq \m_P^2=(X_1,\ldots,X_n)^2$, a prime ideal. We denote $x_i$ to be the image of $X_i$.

Notice that in this situation, the integral closure $\overline{R}$ of $R$ in its fraction field, $Q=\operatorname{Frac}(R)$, is a Discrete Valuation Ring ($DVR$) by \cite[Theorem 4.3.4]{SwansonHuneke}. Thus, henceforth we fix a uniformizing parameter $t$ and denote $\overline{R}=\k\ps{t}$. Using the inclusion $R\subseteq \overline{R}$, we can write every element of $R$ as a power series in $t$. Thus, $R=\ps{\alpha_1t^{a_1}, \ldots, \alpha_nt^{a_n}}$ (i.e., $x_i=\alpha_it^{a_i}$) where each $\alpha_i$ is a unit in $\k\ps{t}$. We also arrange  $a_1<a_2<\dots<a_n$. 

For each power series $p(t)=\sum_{i=0}^\infty c_it^{i}$, we define the order valuation $v(p(t)):=\min\{i\mid c_i\neq 0\}$. When $p(t)=0$, define $v(0)=\infty$. Thus, $v(x_i)=a_i$ for each $i$. We denote $v(R)$ to be the valuation semigroup of $R$. Notice that this extends to a discrete valuation on $Q$ by setting $v(p/q)=v(p)-v(q)$, where $p,q\in R$, $q\neq 0$. 

For each fractional ideal $I$, i.e., a finitely generated $R$-submodule of $Q$, we denote $$v(I):=\min\{v(y)\mid y\in I\}.$$
Also, for each fractional ideal $I$, we denote $I^{-1}:=R:_Q I=\{y\in Q\mid yI\subseteq R\}$.

One of the ideals which is important for our purposes is the conductor $\cC_R$. It is defined to be $\cC_R=R:_Q\overline{R}=\{\alpha\in Q\mid \alpha \overline{R}\subseteq R\}$. One can show that it is the largest common ideal of $R$ and $\overline R $. 
Since $\ds \overline{R}=k\ps{t}$, we have that $\cC_R=(t^i)_{i\geq c_R}\overline{R}$. Here $c_R$ is the smallest integer such that $t^{c_R-1}\not \in R$, and $t^{c_R+i}\in R$ for all $i\geq 0$. It is clear from this  discussion that there cannot be any element $r\in R$, such that $v(r)=c_R-1$.  Since $\ds \overline{R}$ is finitely generated over $R$ (\cite[Theorem 4.3.4]{SwansonHuneke}),  $\cC_R\neq 0$, and it is never equal to $R$ unless $R$ is regular.

\subsection{Universally Finite Module of Differentials}

\begin{definition}\label{defmoduleofdiff}
	Let $(R,\m,\k), P, I$ be as above. Let $I=(f_1,\dots,f_m)$ where $f_j\in P=\k\ps{X_1,\dots,X_n}, n\geq 2$. Then the \textit{universally finite module of differentials over $\k$}, denoted by $\Omega_{R}$, has a (minimal) presentation given as follows:
	$$R^{m}\xrightarrow{\left[\frac{\partial f_j}{\partial x_i}\right]} R^n\to \Omega_{ R}\to 0$$ where $\left[\frac{\partial f_j}{\partial x_i}\right]$ is the Jacobian matrix of $I$, with entries in $R$. 
\end{definition}
For further details regarding the module of differentials, we refer the reader to the excellent book by E. Kunz \cite{Kunzbook}. 

\subsection{Quasihomogeneous Rings}
\begin{definition}(\cite[Satz 9.8]{scheja1970differentialmoduln}, \cite[Definition 3]{Berger_article})\label{defquasihomogeneous}
    Let $(R,\m,\k)$ be a complete local one dimensional non-regular domain which is a $\k$-algebra where $\k$ is algebraically closed of characteristic $0$. Let $\Omega_R$ denote the module of differentials as in \Cref{defmoduleofdiff}. Then $R$ is called \textbf{quasihomogeneous} if there exists an exact sequence  $$\Omega_R\to \m\to 0.$$
\end{definition}
Notice that the hypotheses on $R$ in \Cref{defquasihomogeneous} can be relaxed to quite an extent and the surjection from $\Omega_R$ to $\m$ is the essential part (for instance, one may only require $R$ to be reduced with its $\m$-adic completion being a domain). However, for the purposes of this article, we stick to the above definition. 

\begin{example}
    Let $R=\mathbb{C}\ps{X,Y,Z}/(XZ-Y^2,X^3-Z^2)\cong \CC\ps{t^4,t^5,t^6}$ is quasihomogeneous. First notice that $\Omega_R=RdX\oplus RdY\oplus RdZ/U$ where $U$ is the submodule spanned by $zdX-2ydY+xdZ,3x^2dX-2zdZ$. The map $$dX\to 4x, dY\to 5y, dZ\to 6z$$ defines an $R$-linear surjection to the maximal ideal $(x,y,z)$ of $R$.
\end{example}
\begin{remark}\label{remark on quasi}
For rings in the above situation, the map $R$ to $R$ which is given by $x\mapsto \text{deg}(x)x$ for any homogeneous element $x$ (known as the Euler derivation map) lifts to a surjection from $\Omega_R$ to $\m$ (known as the Euler homomorphism) (see the discussion following \cite[Theorem 2.3]{herzog1994module} and also the proof of \cite[Theorem 3.7]{herzog1994module}). 

Rings of the form $\k\ps{t^{b_1},\dots,t^{b_n}}$ are all quasihomogeneous. To see this, first fix a presentation $\k\ps{t^{b_1},\dots,t^{b_n}}\cong \k\ps{X_1,\dots,X_n}/J$. Now we can define the Euler homomorphism $\Omega_R\rightarrow R $ as $dX_i\rightarrow b_ix_i$ which leads to the surjection $\Omega_R\twoheadrightarrow\m$.
\end{remark}



\section{Quasihomogeneity via order valuation on units}
In this section, we prove our first main result. we recall the following notion of order valuation on the units which was utilised to solve partial cases of a long standing conjecture, in a recent work, see \cite[Notation 4.1, Theorem 4.2]{maitra2023valuations}.

\subsection{Order Valuation of Units}
The units in $\overline{R}$ are of the form $\alpha=\sum_{i=0}^\infty c_it^i$ with $c_0\neq 0, c_i\in \k$. Let $\alpha_j=\sum_{i=0}^\infty c_{ji}t^i$ be such a unit. Then the order valuation of $\alpha_j$ is given by 
\begin{align*}
		o(\alpha_j)=v(\alpha_j-c_{j0})=\min_{i\geq 1}\{i~|~c_{ji}\neq 0 \}.
	\end{align*}
 \begin{theorem}\label{quasihomogeneous and valuations}
	Let $(R,\m,\k)$ be a non-regular, complete, local one dimensional domain which is a $\k$-algebra with $\k$ algebraically closed of characteristic $0$. Let $\overline{R}=\k\ps{t}$ with the conductor of $R$ given by $\mathfrak{C}_R=(t^{c_R})\overline{R}$. Writing $R=\k\ps{\alpha_1t^{a_1},\ldots, \alpha_nt^{a_n}}$ with $a_1<a_2<\cdots<a_n$, set 
	\begin{align*}
	o(\alpha_r)=\min_{1\leq i\leq n}\{o(\alpha_i) \}, \qquad a=\min_{j\neq r}\{a_j\}.
	\end{align*}
	If $o(\alpha_r)+a\geq c_R$, then $R$ is quasihomogeneous.
\end{theorem}
\begin{proof}
As before, we write $x_i=\alpha_it^{a_i}$.  Now assume that $r\geq 2$. Then $a=a_1$. Choose any $1\leq j\leq n$. Then $o(\alpha_j)+a_j\geq o(\alpha_r)+a\geq c_R$. Now $$x_j=\alpha_jt^{a_j}=\alpha_{j0}t^{a_j}+t^{o(\alpha_j)+a_j}h$$ where $h\in \overline{R}$. Since $o(\alpha_j)+a_j\geq c_R$, $t^{o(\alpha_j)+a_j}\in \cC_R\subseteq R$. Thus, $t^{a_j}=\frac{1}{\alpha_{j0}}(x_j-t^{o(\alpha_j)+a_j}h)\in R$. Thus, $R=\k\ps{t^{a_1},\ldots,t^{a_n}}$ and the quasihomogeneity of $R$ follows from \Cref{remark on quasi}.


Finally suppose that $r=1$. Then $a=a_2$. Thus, we get that $o(\alpha_j)+a_j\geq o(\alpha_1)+a_2\geq c_R$ for all $j\geq 2$. Also, we may assume that $\alpha_1=1+\sum_{k\geq 1}\alpha_{1k}t^k$ where $\alpha_{1k}\in\k$. Now let $\beta=\alpha_1^{1/a_1}=1+\sum_{k\geq 1}\beta_kt^{o(\alpha_1)+k-1},\beta_k\in\k$, a unit in $\overline{R}$ (such a $\beta$ exists due to Hensel's Lemma; see for example \cite[Section 2]{maitra2023valuations}). We change the parameter to 
\begin{align*}
s=\beta t =    t+\beta_1t^{o(\alpha_1)+1}+\beta_2t^{o(\alpha_1)+2}+\cdots
\end{align*}

We now identify $R$ inside $\k\ps{s}$. Notice that $s^{c_R-1}\not \in R$ but $s^{c_R+i}\in R$ for all $i\geq 0$ using the relation $s=\beta t$. In other words, the conductor valuation stays the same. We denote the ring $R$ as $\k\ps{x_1,\ldots, x_n}$ where each $x_i$ has a representation both in terms of $s$ and $t$, depending on the choice of the parameter. Notice that both $s^{c_R+i}$ and $t^{c_R+i}$ can be written purely in terms of $x_1,\ldots,x_n$ and consequently $t^{c_R+i}$ can be written purely in terms of $s$ (with valuation equal to $c_R+i$). 

Our goal is to write $t=\gamma s$ for some $\gamma\in\k\ps{s}$ which establishes that $\k\ps{t}=\k\ps{s}$. Notice that
\begin{align*}
s&=\beta t =    t+\beta_1t^{o(\alpha_1)+1}+\beta_2t^{o(\alpha_1)+2}+\cdots\\
\beta_1s^{o(\alpha_1)+1}&=\beta_1t^{o(\alpha_1)+1}+\beta_1((o(\alpha_1)+1)\beta_1)t^{2o(\alpha_1)+1}+\cdots.
\end{align*}
Thus
\begin{align}\label{theorem 3.1 eq1}
  s-\beta_1s^{o(\alpha_1)+1}&=(t+  \beta_2t^{o(\alpha_1)+2}+\cdots) -(\beta_1((o(\alpha_1)+1)\beta_1)t^{2o(\alpha_1)+1}+\cdots)\\
  &=\beta' t\text{ for }\beta'\in\k\ps{t}.\nonumber
\end{align}
Notice that $o(\beta')$ is at least $o(\alpha_1)+1$ or $2o(\alpha_1)$. In either case, we see that $o(\beta')\geq o(\alpha_1)+1$. We keep repeating this process of subtracting terms involving higher powers of $t$ in the expression \eqref{theorem 3.1 eq1} to arrive at
\begin{align*}
    s-f(s)=t+\sum_{k\geq 0}\eta_k t^{c_R+k}\text{ where }f(s)\in (s) \k\ps{s}, \eta_k\in\k.
\end{align*}
Now as explained in the paragraphs above, $t^{c_R+k}=g_k(s)\in (s)\k\ps{s}$ (i.e., they can be written purely in terms of $s$). Thus we arrive at
\begin{align*}
    s-f(s)-\sum_{k\geq 0}\eta_kg_k(s)=t.
\end{align*}
Thus $t=\gamma s,\gamma\in\k\ps{s}$. Now notice that $o(\gamma)=o(\alpha_1)=o(\beta)$.

Now $s^{a_1}=(\beta t)^{a_1}=\alpha_1t^{a_1}=x_1$. Also
\begin{align*}
    x_i=\alpha_it^{a_i}&=\alpha_{i0}t^{a_i}+\alpha_{i1}t^{o(\alpha_i)+a_i}+\alpha_{i2}t^{o(\alpha_i)+a_i+1}+\cdots\\
    &=\alpha_{i0}(\gamma s)^{a_i}+\alpha_{i1}(\gamma s)^{o(\alpha_i)+a_i}+\alpha_{i2}(\gamma s)^{o(\alpha_i)+a_i+1}+\cdots\\
    &=\alpha'_{i0}s^{a_i}+\alpha'_{i1}s^{o(\alpha_1)+a_i}+\cdots\\
    &=\alpha_i's^{a_i}.
\end{align*}
Notice that for all $i\geq 2$, $o(\alpha_1)+a_i\geq o(\alpha_1)+a_2\geq c_R$. Thus following the same method as in the beginning of this proof, we can rewrite  $x_i=s^{a_i}$. Since this can be performed for all $x_i$, we have $R=\k\ps{s^{a_1},\dots,s^{a_n}}$. Thus quasihomogeneity of $R$ now follows from \Cref{remark on quasi}.
\end{proof}
\begin{example}
	Let $R=\k\ps{t^4+t^5,t^7,t^8,t^9}$. Then the conductor $\mathfrak{C}_R=(t^7)\overline{R}$. Here $o(\alpha_1)=1$ and $a=7$. Since $o(\alpha_1)+a\geq c=7$,  $R$ is a quasihomogeneous ring. Here $R=\k\ps{s^{4},s^7,s^8,s^9}$.
\end{example}
\begin{example}
    Let $R=\k\ps{t^5,t^6,t^8+t^9}$. Then \cite[Page 207,(1)]{Waldi79} shows that $R$ is not quasihomogeneous. Notice that $o(\alpha_3)+a_1=1+5=6<10=c_R$.
\end{example}
\section{Quasihomogeneity via trace ideal of $\Omega_R$}
We recall the following notions and results which will be crucial in establishing our second main result. 

\begin{definition}
   Let $R$ be a local domain with fraction field $Q$. Then the \textit{rank} of any module $M$ is defined to be $\Rank(M):=\dim_Q (M\otimes_ R Q)$. 
\end{definition}

\begin{definition}
Let $M$ be a finitely generated $R$-module. Then the \textit{trace ideal} of $M$, denoted $\text{tr}_R(M)$, is the ideal $\sum \alpha (M)$ where $\alpha $ ranges over $M^*:=\Hom_R(M,R)$. An ideal $I$ is said to be a trace ideal if $I=\text{tr}_R(M)$ for some module $M$. 
\end{definition}
Clearly, $I\subseteq \Tr(I)$. For further details on trace ideals and their properties, we refer the reader to resources such as \cite[Proposition 2.4]{kobayashi2019rings} and \cite{LindoTrace}. we shall need the following properties in this work.
\begin{lemma}\label{traceprop}
    Let $(R,\m,\k)$ be a one dimensional non-regular complete local domain which is a $\k$-algebra, where $\k$ is algebraically closed of characteristic $0$. Let $\overline{R}=\k\ps{t}$ with $v$ the order valuation. Let $I$ be a fractional ideal of $R$, i.e., an $R$-submodule of $Q$. Then the following statements hold.
    \begin{enumerate}

\item For any $R$-module $N$, we have that $\Tr(\Tr(N))=\Tr(N)$. 
    \item  Suppose $J$ is another fractional ideal that is isomorphic to $I$. Then $\text{tr}_R(I)=\text{tr}_R(J)$
        \item $\Tr(I)=II^{-1}$. 

        \item $v(\Tr(I))=v(I)+v(I^{-1})$.

        \item Suppose $M$ is a finitely generated $R$-module of rank one. Suppose $f:M\to R$ is any non-zero map and let $I=f(M)$. Then $\Tr(M)=\Tr(I)$.
    \end{enumerate}

\end{lemma}

\begin{proof}
    Suppose $\phi:I\to J$ be the isomorphism. Then by composing any map $\alpha:J\to R$ with $\phi$ we get a map from $I\to R$. Thus, $\Tr(J)\subseteq \Tr(I)$. Now a symmetric argument finishes the proof of $(2)$. Statement $(1)$ also follows a similar reasoning.

    Statement $(3)$ appears as \cite[Proposition 2.4(2)]{kobayashi2019rings} whereas statement $(5)$ is \cite[Proposition 3.5]{maitra2020partial}. 

    Statement $(4)$ is immediate from $(3)$.
\end{proof}

\subsection{Fractional Ideal of Derivatives}
Let $R=\k\ps{x_1,\ldots,x_n}, n\geq 2$ where each $x_i=\alpha_it^{a_i}$ where $a_1<a_2<\ldots<a_n$ as in our original setup. Let $\Omega_R$ be the module of differentials. From \cite[Section 2.3]{huneke2021torsion} and ignoring $dt$, we see that $\Omega_R$ surjects to the $R$-span of $x_i'(t):=\frac{dx_i}{dt}$. This is clearly a fractional ideal and we call it the fractional ideal of derivatives. Henceforth we shall use the following notation -- $$\displaystyle \mathcal{D}:=Rx_1'(t)+Rx_2'(t)+\cdots +Rx_n'(t).$$ Notice that some power series of high enough valuation multiplies $\mathcal{D}$ into $R$ and hence we get an ideal of $R$ to which $\Omega_R$ surjects. This ideal of course is isomorphic to $\mathcal{D}$ as the map is simply by multiplication by a non-zero element of $Q$. For instance, $t^{c_R}\mathcal{D}$ is such an ideal, where $c_R=v(\cC_R)$.

\subsection{The $\h(\cdot)$ Invariant}
\begin{definition}\cite[Definition 2.1]{maitra2020partial}
Let $R$ be a one dimensional local domain and $M$ be a finitely generated $R$-module. Then $$\h(M)=\min\{\ell(R/I)\mid \text{there exists a homomorphism}~ \phi:M\to R~\text{with~}\phi(M)=I\}$$
where $\ell(\cdot)$ denotes length. 
\end{definition}

Notice that $\h(\cdot)$ is a non-negative integer or takes the value $\infty$.

\begin{lemma} \cite[Theorem 5.2]{maitra2020partial}\label{hunexplicit}
    Let $R,\Omega_R,\mathcal{D}$ be as discussed above. Then $$\h(\Omega_R)=\ell(\overline{R}/\mathcal{D})-\ell(\overline{R}/R)+v(\mathcal{D}^{-1})\leq v(\mathcal{D}^{-1}).$$
\end{lemma}

\begin{theorem}
    Let $(R,\m,\k)$ be a non-regular complete local one dimensional domain with $\k$ algebraically closed of characteristic $0$. Let $\Omega_R$ be the module of differentials and $\overline{R}=\k\ps{t}$ with valuation function $v$. Then $R$ is quasihomogeneous if and only if $$v(\Tr(\Omega_R))=v(\m).$$
\end{theorem}

\begin{proof}
Throughout the proof, we assume that $R=\k\ps{x_1,\ldots, x_n}, n\geq 2$ with $x_i=\alpha_it^{a_i}$, $a_1<a_2<\cdots<a_n$.

    Assume that $R$ is quasihomogeneous. So there exists a surjection from $\Omega_R$ to $\m$. Since $\Rank(\Omega_R)=1$, we have that $\Tr(\Omega_R)=\Tr(\m)$ by \Cref{traceprop}$(5)$. Now $\m\subseteq \Tr(\m)$ and if the latter is equal to $R$, then $\m$ has a free direct summand by \cite[Proposition 2.8(iii)]{LindoTrace}, contradicting the non-regularity assumption of $R$. Thus $\Tr(\m)=\m$ and hence $v(\Tr(\Omega_R))=v(\m)$. 

    Conversely, suppose that $v(\Tr(\Omega_R))=v(\m)=v(x_1)$. In order to prove that $R$ is quasihomogeneous, it is enough to show that $\hun=1$: for suppose $1=\hun=\ell(R/I)$ for some ideal $I$ which is a surjective image of $\Omega_R$.  Now $I\subseteq \m\subseteq R$ and hence $1=\ell(R/I)\geq \ell(R/\m)=1$ which implies that $I=\m$.
    
    By \Cref{traceprop}, we get that $\Tr(\Omega_R)=\Tr(\mathcal{D})=\mathcal{D}\mathcal{D}^{-1}$ and hence $v(\Tr(\Omega_R))=v(\mathcal{D})+v(\mathcal{D}^{-1})$. Since $\k$ is of characteristic $0$, we get that $v(\mathcal{D})=v(x_1'(t))=v(x_1)-1=v(\m)-1$. Using this and the hypothesis, we get that $v(\mathcal{D}^{-1})=v(\Tr(\Omega_R))-v(\mathcal{D})=v(\m)-v(\m)+1=1$. 

    Thus, using \Cref{hunexplicit}, we get that $$\hun=\ell(\overline{R}/\mathcal{D})-\ell(\overline{R}/R)+1\leq 1.$$ 
    If $\hun=0$, then $R$ is regular (for instance, see \cite[Remark 4.2]{maitra2020partial}), a contradiction. Thus $\hun=1$, thereby finishing the proof. 
\end{proof}

\bibliographystyle{plain}
\bibliography{references}

\begin{thebibliography}{10}

\bibitem{Berger_article}
Robert~W Berger.
\newblock Report on the {T}orsion of {D}ifferential {M}odule of an {A}lgebraic
  {C}urve.
\newblock In {\em Algebraic Geometry and its Applications}, pages 285--303.
  Springer, 1994.

\bibitem{herzog1994module}
J{\"u}rgen Herzog.
\newblock The module of differentials.
\newblock In {\em lecture notes from Workshop on Commutative Algebra and its
  Relation to Combinatorics and Computer Algebra (Trieste, 1994)}, 1994.

\bibitem{huneke2021torsion}
Craig Huneke, Sarasij Maitra, and Vivek Mukundan.
\newblock Torsion in differentials and berger’s conjecture.
\newblock {\em Research in the Mathematical Sciences}, 8(4):1--15, 2021.

\bibitem{kobayashi2019rings}
Toshinori Kobayashi and Ryo Takahashi.
\newblock Rings whose ideals are isomorphic to trace ideals.
\newblock {\em Mathematische Nachrichten}, 292(10):2252--2261, 2019.

\bibitem{Kunzbook}
Ernst Kunz.
\newblock {\em K{\"a}hler differentials}.
\newblock Advanced Lectures in Mathematics. Friedr. Vieweg \& Sohn,
  Braunschweig, 1986.

\bibitem{KunzRuppert77}
Ernst Kunz and Walter Ruppert.
\newblock Quasihomogene {S}ingularit\"{a}ten algebraischer {K}urven.
\newblock {\em Manuscripta Math.}, 22(1):47--61, 1977.

\bibitem{LindoTrace}
Haydee Lindo.
\newblock Trace ideals and centers of endomorphism rings of modules over
  commutative rings.
\newblock {\em J. Algebra}, 482:102--130, 2017.

\bibitem{maitra2020partial}
Sarasij Maitra.
\newblock Partial trace ideals and {B}erger's conjecture.
\newblock {\em Journal of Algebra}, 598:1--23, 2022.

\bibitem{maitra2023valuations}
Sarasij Maitra and Vivek Mukundan.
\newblock Valuations and nonzero torsion in module of differentials.
\newblock {\em Bulletin des Sciences Math{\'e}matiques}, page 103287, 2023.

\bibitem{Saito71}
Kyoji Saito.
\newblock Quasihomogene isolierte {S}ingularit\"{a}ten von {H}yperfl\"{a}chen.
\newblock {\em Invent. Math.}, 14:123--142, 1971.

\bibitem{scheja1970differentialmoduln}
G{\"u}nter Scheja.
\newblock Differentialmoduln lokaler analytischer {A}lgebren, {S}chriftenreihe
  {M}ath.
\newblock {\em Inst. Univ. Fribourg, Univ. Fribourg, Switzerland}, 1970.

\bibitem{SwansonHuneke}
Irena Swanson and Craig Huneke.
\newblock {\em Integral closure of ideals, rings, and modules}, volume 336 of
  {\em London Mathematical Society Lecture Note Series}.
\newblock Cambridge University Press, Cambridge, 2006.

\bibitem{Waldi79}
Rolf Waldi.
\newblock Deformation von {G}orenstein-{S}ingularit\"{a}ten der {K}odimension
  {$3$}.
\newblock {\em Math. Ann.}, 242(3):201--208, 1979.

\bibitem{Zariski66}
Oscar Zariski.
\newblock Characterization of plane algebroid curves whose module of
  differentials has maximum torsion.
\newblock {\em Proc. Nat. Acad. Sci. U.S.A.}, 56:781--786, 1966.

\end{thebibliography}
 \end{document}